\documentclass[11pt]{amsart}
\usepackage{amsfonts,fancyhdr}
\usepackage{amsmath,amssymb,amsthm,latexsym,graphicx,graphics}

\newtheorem{theorem}{Theorem}[section]

\newtheorem{lema}[theorem]{Lemma}

\newtheorem{remark}[theorem]{Remark}

\newtheorem{thm}{Theorem}

\usepackage{url,hyperref}
\usepackage{euscript,color}

\hypersetup{
    bookmarks=false,         
    unicode=false,          
    pdftoolbar=true,        
    pdfmenubar=true,        
    pdffitwindow=false,     
    pdfstartview={FitH},    
    pdftitle={My title},    
    pdfauthor={Author},     
    pdfsubject={Subject},   
    pdfcreator={Creator},   
    pdfproducer={Producer}, 
    pdfkeywords={keyword1} {key2} {key3}, 
    pdfnewwindow=true,      
    colorlinks=true,       
    linkcolor=black,          
    citecolor=green,        
    filecolor=magenta,      
    urlcolor=red         
}

\def\meti{\left\langle}
\def\metd{\right\rangle}

\begin{document}

\title[Codimension reduction in symmetric spaces]
{Codimension reduction in symmetric spaces}

\author[A. J. Di Scala]{Antonio J. Di Scala}
\author[F. Vittone]{Francisco Vittone}

\thanks{This work was partially supported by ERASMUS MUNDUS ACTION 2 programme, through the EUROTANGO II Research Fellowship. The second author would like to thanks Politecnico di Torino for the hospitality during his research stay. }

\date{\today}

\maketitle

\begin{abstract}
In this paper we give a short geometric proof of a generalization of a well-known result about reduction of codimension for submanifolds of Riemannian symmetric spaces.
\end{abstract}

\section{Introduction}
The goal of this paper is to give a short geometric proof of the following generalization of the reduction of codimension theorem for submanifolds of space forms \cite[page 339]{Er}:

\begin{thm} Let $M$ be a submanifold of a symmetric space $\mathbb{S}$ and let $\nu(M)$ be its normal bundle. Assume that there exists a $\nabla^{\perp}$-parallel subbundle $\mathrm{V} \subset \nu(M)$ containing the first normal space, i.e. $\mathrm{N}^1 \subset \mathrm{V}$ where $\mathrm{N}^1 = \alpha(TM \times TM)$. If $TM \oplus \mathrm{V}$ is invariant by the curvature tensor of  $\mathbb{S}$ then there exists a totally geodesic submanifold of $\mathbb{S}$ of dimension equal to $\mathrm{rank}(TM \oplus \mathrm{V})$ containing $M$.
\label{teorema}
\end{thm}

As a corollary of this result one can obtain several well-known special cases \cite{Ce},  \cite{CHL}, \cite{Ka}, \cite{Ka2}, \cite{KiPa}, \cite{KwPa}, \cite{Ok} .

The hypothesis about the curvature invariance of $TM \oplus \mathrm{V}$ is redundant if $\mathbb{S}$ is a space form. We will give an example showing that such condition can not be omitted in general, see Section \ref{FR}.

Our proof of the above theorem was mainly inspired by the proof, due to C. Olmos,
of the existence theorem of a totally geodesic submanifold with prescribed tangent space usually attributed to E. Cartan, see \cite[Theorem 8.3.1., page 231]{BCO}.
Olmos' proof is based in Lemma 8.3.2 in  \cite[page 232]{BCO}. We will need to use a slightly different version of this lemma which involve parallel translation along piece-wise smooth curves instead of smooth curves. We include it in an appendix, with a sketch of its proof, for the sake of completeness.

Theorem \ref{teorema} does not hold for submanifolds of locally symmetric spaces. The problem is that under the same hypothesis,  the ``totally geodesic submanifold" containing $M$ may intersect itself, see example in Section \ref{FR}. One can prove a slightly different version of this theorem for locally symmetric spaces by either assuming that the submanifold $M$ is embedded or allowing totally geodesic immersions (not necessarily $1-1$) instead of totally geodesic submanifolds.

Finally, we want to point out that Theorem \ref{teorema} can be obtained, besides our proof, by following two other different approaches.
The first one makes use of the Grassmann bundle theory and the integration theory of differentiable distributions, see \cite[Prop. 3, page 90]{JeRe}.
The second one is based on a generalization of the classical theorem of existence and uniqueness of isometric immersions into space forms, see \cite{EsTr}.
\section{Basic definitions}

We will say that a Riemannian manifold $M$ is a submanifold of a Riemannian manifold $\mathbb{S}$ if there is a $1-1$ isometric immersion $f:M\to \mathbb{S}$. In order to simplify the notation, we shall assume that $M$ is a subset of $\mathbb{S}$, eventually endowed with a different topology,  and $f$ is the inclusion map. If in addition $M$ has the induced topology from $\mathbb{S}$ we say  that $M$ is an embedded submanifold.

We identify the tangent space to $M$ at a point $p$ with a subspace of $T_p\mathbb{S}$ and consider the orthogonal splitting $T_p \mathbb{S}=T_pM\oplus \nu_p M$. Here $\nu_pM$ is the normal space and $\nu (M)$ will denote the normal bundle of $M$.

We denote by $\overline{\nabla}$ the Levi-Civita connection of $\mathbb{S}$ and by  $\nabla$  and $\nabla^{\bot}$ the Levi-Civita and the normal connections of $M$ respectively.  Let $\alpha$ and $A$ be the second fundamental form and shape operator of $M$ respectively. They are defined taking tangent and normal components by the Gauss and Codazzi formulas
\begin{equation}
\overline{\nabla}_X Y=\nabla_X Y+ \alpha(X,Y),\qquad \overline{\nabla}_X\xi=-A_{\xi}X+\nabla^{\bot}_X\xi
\label{gausscodazzi}
\end{equation}
and related by $\meti\alpha(X,Y),\xi\metd=\meti A_{\xi}X,Y\metd$, for any tangent vector fields $X$ and $Y$ to $M$ and any normal vector field $\xi$.


\section{Proof of Theorem \ref{teorema}}


First notice that it suffices to prove the theorem locally around each point. Namely, to show the existence of a totally geodesic submanifold $N_p$ of $\mathbb{S}$ containing a neighbourhood $U$ of $p$ in $M$ whose tangent space is $T_q N_p=T_q M\oplus \mathrm{V}_q $ for all $q\in U$.  Indeed, the global result follows since a complete totally geodesic submanifold of a symmetric space $\mathbb{S}$ with a prescribed tangent space is unique as a global object \cite[Lemma 2, page 235]{KoNo}.

So we may assume that $M$ is small enough so that the normal exponential  map $\exp^{\bot}:\mathrm{V}_0\to \mathbb{S}$ is an immersion from a small neighbourhood $\mathrm{V}_0$ of the zero section of $\mathrm{V}$.

Set $N=\exp^{\bot}(\mathrm{V}_0)$.  Since $M$ is the image of the zero section of $\mathrm{V}$ we get that $M$ is a submanifold of $N$. Now we are going to prove that $N$ is a totally geodesic submanifold by a similar argument as in the proof of Theorem 8.3.1 in  \cite[page 231]{BCO}.  It will suffice to prove that the parallel transport in $\mathbb{S}$ along any curve in $N$ preserves the tangent bundle $TN$.
To do this we will fix a point $p\in M$ and we will show the following two properties:
\begin{itemize}
\item[i)] \label{1} for any point $q$ in $N$, there is a curve $\gamma$ joining $p$ with $q$ such that the parallel transport along $\gamma$ in $\mathbb{S}$ of $T_pN$ is $T_{\gamma(t)}N$;
\item[ii)] the tangent space $T_pN$ is preserved by parallel transport in $\mathbb{S}$ along any loop in $N$  based at $p$.
\end{itemize}

In order to prove i), we will start showing that $TN$ is parallel with respect to the connection $\overline{\nabla}$ of $\mathbb{S}$ in directions tangent to $M$.

Since 	$TN_{|M}=TM\oplus \mathrm{V}$, a section $X$ of $TN_{|M}$ splits as  $X=X_1+X_2$,  with $X_1\in TM$ and $X_2\in \mathrm{V}$. So if $v\in T_pM$, then $$\overline{\nabla}_v X=\nabla_v X_1+\alpha(v,X_1)+\nabla^{\bot}_{v}X_2-A_{X_2}v.$$
This shows that $\overline{\nabla}_v X$ belongs to $TN$
since $\alpha(v,X_1)\in \mathrm{N}^{1}\subset \mathrm{V}$ and $\mathrm{V}$ is parallel with respect to the normal connection of $M$.

The second step is to prove that $TN$ moves parallel along any normal geodesic $\gamma(t)=\exp_p(t\xi_p)$ for $p\in M$ and $\xi_p\in \mathrm{V}_{0}$. 	
 Observe that the tangent spaces to $N$ along $\gamma$ are generated by the Jacobi fields $J(t)$ along $\gamma(t)$ with initial conditions $J(0)\in T_{\gamma(0)}M$ and $J'(0)\in \mathrm{V}$.

Denote by $\mathrm{W}_t$ the parallel transport of $T_p N$ along $\gamma$ from $\gamma(0)$ to $\gamma(t)$.

Let $J(t)$ be any Jacobi vector field along $\gamma$ with $J(0)\in T_{\gamma(0)}M$ and $J'(0)\in \mathrm{V}_{\gamma(0)}$. Since $TM\oplus \mathrm{V}$ is invariant under the curvature tensor of the symmetric space $\mathbb{S}$  one gets that $J(t)\in \mathrm{W}_t$ for every $t$. Indeed,  $\mathrm{W}_t$ is curvature invariant and so the Jacobi equation can be solved in $\mathrm{W}_t$. This shows that $T_{\gamma(t)}N\subset \mathrm{W}_t$, hence $\mathrm{W}_t=T_{\gamma(t)}N$, since both are linear spaces of the same dimension.

Now, if $q$ is any point in $N$, there exists a point $q_0$ in $M$ and a normal vector $\xi_{q_0}\in \mathrm{V}_0$ such that $q=\exp^{\bot}(\xi_{q_0})$. From the above discussion, any curve in $M$ connecting $p$ to $q_0$ followed by a normal geodesic from $q_0$ in the direction of $\xi_0$ gives one curve joining $p$ with $q$ satisfying i).

\vspace{0.5cm}
Now we prove ii) by using Lemma \ref{lemacarlos} in appendix (cf. \cite[Lemma 8.3.2, page 232]{BCO}.  Let $c(s)$ be any loop in $N$ based at $p\in M$. There exists a loop $\hat{c}(s)$ in $M$ based at $p$ and a normal vector field $\xi(s)\in \mathrm{V}_{\hat{c}(s)}$ along $\hat{c}$ such that $c(s)=\exp^{\bot}(\xi(s))$.

For each $s\in I$, define the transformation $\tau(s)\in SO(T_p \mathbb{S})$ obtained by $\overline{\nabla}$-parallel transport along the curve $c$ from $p=c(0)$ to $c(s)$, then along the normal geodesic $\gamma_{s}(t)=\exp^{\bot}(t\xi(s))$ backwards from $\gamma_s(1)$ to $\gamma_s(0)$ and finally backwards along $\hat{c}$, from $\hat{c}(s)$ to $\hat{c}(0)=p$.

Observe that $\tau(0)$ is the identity transformation of $T_p \mathbb{S}$ and $\tau(1)$ is the $\overline{\nabla}$-parallel transport along the loop $c$ followed by the $\overline{\nabla}$-parallel transport along the loop $\hat{c}^{-1}$.

Consider now the function $f:I\times I\to N$ defined by
$$f(s,t)=\left\{\begin{array}{cl}
\hat{c}(2st) & \text{ if }0\leq t\leq \frac{1}{2},\ s\in I\\
\exp^{\bot}((2t-1)\xi(s)) & \text{ if }\frac{1}{2}\leq t\leq 1,\ s\in I
\end{array}\right.$$
Observe that $f(s,0)=p$ for all $s\in I$ and the transformation $\tau(s)$ defined above is the $\overline{\nabla}$-parallel transport along the curve $t\mapsto f(0,t)$ from $t=0$ to $t=1$, then along the curve $s\mapsto f(s,1)$ from $0$ to $s$, and finally along the curve $t\mapsto f(s,t)$, backwards from $t=1$ to $t=0$.

For each $s\in I$ set $A(s)=\tau'(s)\circ\tau(s)\in\mathfrak{so}(T_p \mathbb{S})$.
We can  apply Lemma \ref{lemacarlos} to obtain that for each $u,\; v\in T_p\mathbb{S}$,
\begin{equation}
\meti A(s)u,v\metd=\int_0^{1}\meti \overline{R}\left(\frac{\partial f}{\partial s}(s,t), \frac{\partial f}{\partial t}(s,t)\right) U_s(t), W_s(t)\metd\; dt
\label{lemacartan}
\end{equation}
where $U_s(t)$ and $W_s(t)$ are $\overline{\nabla}$-parallel vector fields along the curve $t\mapsto f(s,t)$ with $U_s(0)=u$ and $W_s(0)=w$.

Observe that for each fixed $s$, the curve $t\mapsto f(s,t)$ is the concatenation of a curve in $M$ and a normal geodesic. We have seen that for each $s$, the tangent space $T_{f(s,t)}N$ is invariant under $\overline{\nabla}$-parallel transport along these curves and, since $\mathbb{S}$ is symmetric, equation (\ref{lemacartan}) implies that $$\meti A(s)u,v\metd=0$$ for each $u\in T_p N$, $w\in \nu_p N$. That is, $A(s)(T_p N)\subset T_p N$ for all $s\in I$. Since $\tau(s)$ is defined by the system of differential equations $\tau'(s)=A(s)\tau(s)$ and $\tau(0)$ is the identity transformation, one gets that $\tau(s)$ preserves $T_p N$. In particular, $\tau(1)$ preserves $T_p N$, but by construction $\tau(1)$ is the $\overline{\nabla}$-parallel transport along the concatenation of the loops $c$ and $\hat{c}^{-1}$. Since by i) the $\overline{\nabla}$-parallel transport along $\hat{c}$ preserves $T_p N$, we obtain ii). \hfill $\square$

\begin{remark}
The same proof shows that Theorem \ref{teorema} is still true if the ambient space $\mathbb{S}$ is locally symmetric, as long as the submanifold $M$ is embedded.  However if $\mathbb{S}$ is locally symmetric and $M$ is not embedded, the theorem is not true as we will show in the example below.
\end{remark}

\section{Further Remarks}\label{FR}
As we noticed in the Introduction, Theorem \ref{teorema} does not hold under the weaker assumption of the ambient space  $\mathbb{S}$ being locally symmetric even if $\mathbb{S}$ is compact.
For example let $\mathbb{S} = S^1 \times \Sigma $ be the product of a circle $S^1$ with a compact Riemann surfaces $\Sigma$ of genus 2 endowed with the metric of constant negative curvature. It is well-known that there is a self-intersecting geodesic $\gamma$ in $\Sigma$. Then the product $N := I \times \gamma$ is a subset of $\mathbb{S}$ but it is not a submanifold. However $N$ can be regarded as the image of a totally geodesic (non injective!!) immersion. So we can regard $N$ as in Figure 1.
\begin{figure}[h!]
\centering
\vspace{-0.2cm}
\includegraphics[draft=false,width=9cm,height=8cm]{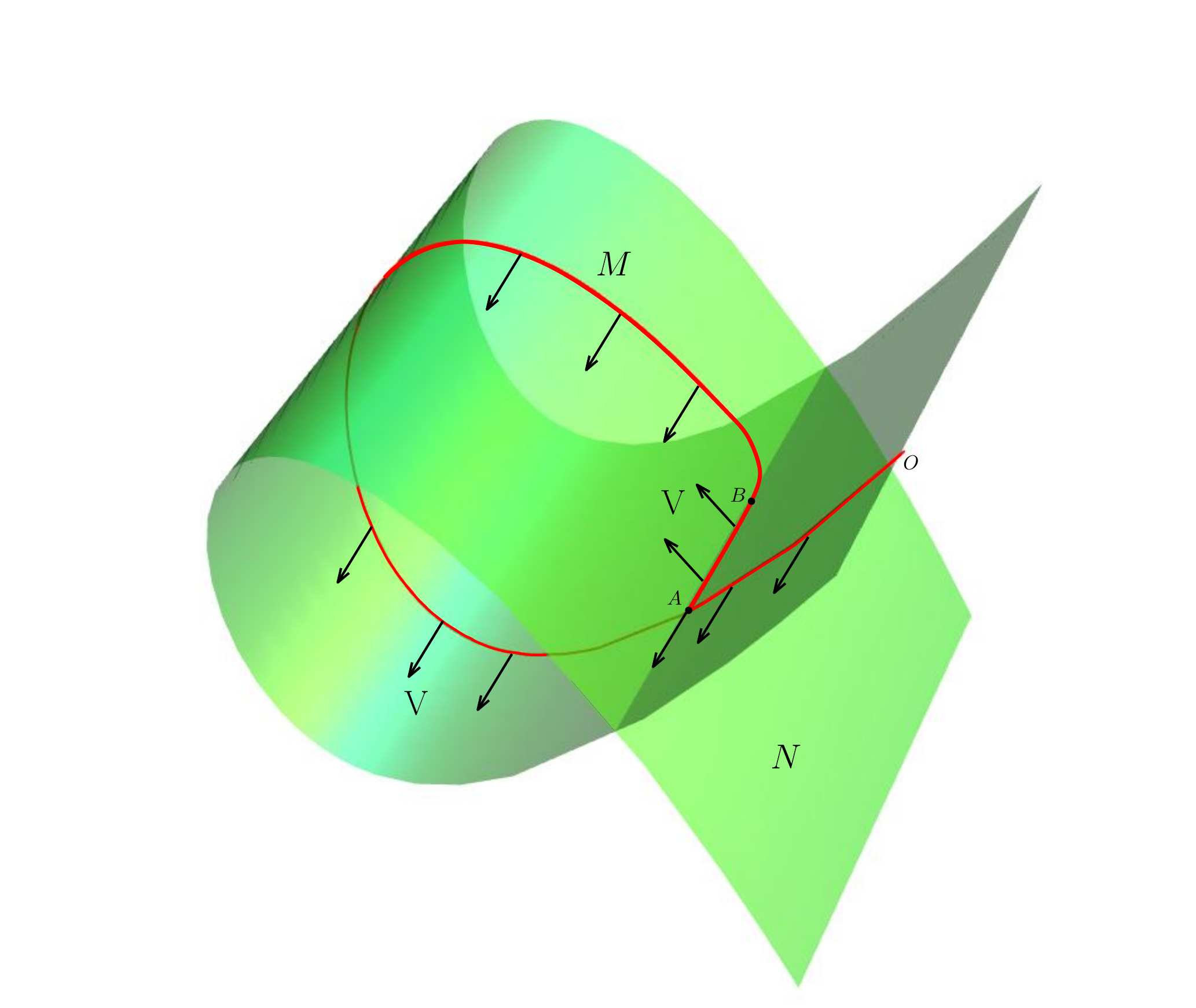}
\caption{}
\end{figure}

Consider the $1-1$ immersed curve $M$ starting at $O$ going through the point $A$ coming back to $B$ and finally approaching the point $A$ along the self-intersection of $N$. The vector field $\mathrm{V}$ is clearly continuous hence it generates a subbundle $\mathrm{V}$ of $\nu(M)$. Since $N$ is the image of a totally geodesic immersion it follows that the subbundle $\mathrm{V}$ is $\nabla^{\perp}$-parallel and contains the first normal space of $M$ as in Theorem \ref{teorema}. Now it is clear that there is not a totally geodesic submanifold containing $M$ as in Theorem \ref{teorema}.
Indeed, such totally geodesic submanifold should be contained in $N$ and self-intersect near the point $A$.
This is so because near $A$, coming from the point $O$, it should contain an open subset of the surface $N$ tangent to $\mathrm{V}$.
But when $M$ approaches $A$ coming from $B$ such totally geodesic submanifold should be contained in the leaf of $N$ tangent to $\mathrm{V}$ which is transversal to the first one. This shows that such `totally geodesic submanifold' intersects itself which is a contradiction since by submanifolds we intend 1-1 immersions.\\

We want also to remark the importance of the hypothesis of $TM\oplus \mathrm{V}$ being curvature invariant.
By using the existence theorem for curves by means of its Frenet-Serret curvatures, see for example \cite[Page 2158, Lemma 4]{Gu},
let $\gamma$ be a regular curve in $\mathbb{C}P^2$ (the complex projective space of dimension 2) with $\kappa_1  = \kappa_2 = 1$ and $\kappa_3 \equiv 0$. Then the mean curvature vector field $H$ of $\gamma$ and its normal derivative $\nabla^{\perp}_{\gamma'(t)} H$ are linearly independent. Since $\kappa_3 \equiv 0$ the rank 2 vector subbundle $\mathrm{V} = \mathrm{span} \{ H , \nabla^{\perp}_{\gamma'(t)} H \} $, which contains the first normal space $\mathrm{N}^1 = \mathrm{span} \{ H \}, $ is $\nabla^{\perp}$-parallel. Since $\mathbb{C}P^2$ has no 3-dimensional totally geodesic submanifolds we conclude that hypothesis of $T \gamma \oplus \mathrm{V}$ being curvature invariant can not be removed from Theorem \ref{teorema}.

\section*{Appendix}

We present here a slight variation of Lemma 8.3.2 in \cite[page 232]{BCO} with a sketch of its proof.

Let $\overline{M}$ be a differentiable manifold and $f:[0,1]\times[0,1]\to \overline{M}$ a continuous map. We say that $f$ is piecewise-smooth if there exist points  $0=t_0<t_1<\cdots<t_n=1$ such that $f_{|[0,1]\times(t_i,t_{i+1})}$ is smooth for $i=0,\cdots, n-1$.

\begin{lema}
Let $\overline{M}$ be a Riemannian manifold and $p\in \overline{M}$. Let $f:[0,1]\times [0,1]\to \overline{M}$ be a piecewise-smooth map with $f(s,0)=p$ for all $s\in [0,1]$. For each $s\in [0,1]$, we define $f_s:[0,1]\to \overline{M}$, $t\mapsto f(s,t)$ and for each $t\in [0,1]$ we define $f^t:[0,1]\to \overline{M}$, $s\mapsto f(s,t)$. For each $s\in [0,1]$, denote by $\tau(s)\in SO(T_p \overline{M})$the orthogonal transformation of $T_p\overline{M}$ obtained by parallel translation along $f_0$ from $p=f_0(0)$ to $f_0(1)=f^1(0)$, then along $f^1$ from $f^1(0)$ to $f^1(s)=f_s(1)$ and finally along $f_s$ from $f_s(1)$ to $f_s(0)=p$. Let $A(s)\in \mathfrak{so}(T_p\overline{M})$ be the skew-symmetric transformation of $T_p\overline{M}$ defined by $A(s)=\tau'(s)\circ \tau(s)^{-1}$ for all $s\in [0,1]$. Then for each $u,w\in T_p\overline{M}$, \begin{equation}
\meti A(s)u,w\metd=\int_0^1\meti \overline{R}\left(\frac{\partial f}{\partial s}(s,t),\frac{\partial f}{\partial t}(s,t)\right) U_s(t), W_s(t)\metd\; dt,
\label{formula}
\end{equation}
where $U_s(t)$ and $W_s(t)$ are the parallel vector fields along $f_s$ with $U_s(0)=u$ and $W_s(0)=w$ respectively.
\label{lemacarlos}
\end{lema}

\begin{proof}
Using the same argument as in the proof of Lemma 8.3.2 in \cite{BCO} one can see that it suffices to prove formula (\ref{formula}) only for $s=0$.
Let $U(s,t)$ be the vector field along $f(s,t)$ obtained by parallel translation of $u$ along $f_0$ from  $p=f_0(0)$ to $f_0(1)=f^1(0)$, then along $f^1$ from $f^1(0)$ to $f^1(s)=f_s(1)$ and finally along $f_s$ from $f_s(1)$ to $f_s(t)$. Then $$U(s,0)=\tau(s)u,\qquad A(0)u=\tau'(0)u=Z(0),$$
where $Z$ is the vector field along $f_0$ defined by $Z(t)=\left(\frac{D}{\partial s} U\right) (0,t)$.  Since the vector field $t\mapsto U(s,t)$ is parallel along $f_s$, $\frac{D}{\partial t} U(s,t)=0$ for $t\in (t_i,t_{i+1})\ i=0,\cdots, n-1$ and so $$Z'(t)=\left(\frac{D}{\partial t}\frac{D}{\partial s} U\right) (0,t)=\overline{R}\left( \frac{\partial f}{\partial t}(0,t),\frac{\partial f}{\partial s} (0,t)\right) U_0(t),  \ t\in I-\{t_i\}_{i=1}^{n-1}.$$
Consider the piecewise smooth function $$g(t)=\meti Z(t), W_0(t)  \metd.$$
For each $t\in I-\{t_i\}_{i=1}^{n-1}$,  \[ g'(t)=\meti Z'(t), W_0(t)\metd =\meti \overline{R}\left( \frac{\partial f}{\partial t}(0,t),\frac{\partial f}{\partial s} (0,t)\right) U_0(t), W_0(t)\metd . \]
Since $g(t)$ is continuous on $[t_i, t_{i+1}]$, $i=0,\cdots, n-1$, we can repeatedly apply Barrow's law and get
\begin{eqnarray*}
\meti A(0)u,w\metd&=&g(0)\\
&=&g(t_1)-\int_0^{t_1}\meti \overline{R}\left( \frac{\partial f}{\partial t}(0,t),\frac{\partial f}{\partial s} (0,t)\right) U_0(t), W_0(t)\metd dt\\
&=& g(t_2)-\sum_{i=0,1}\int_{t_i}^{t_{i+1}}\meti \overline{R}\left( \frac{\partial f}{\partial t}(0,t),\frac{\partial f}{\partial s} (0,t)\right) U_0(t), W_0(t)\metd dt\\
&\vdots&\\
&=& g(1)-\sum_{i=0}^{n-1}\int_{t_i}^{t_{i+1}}\meti \overline{R}\left( \frac{\partial f}{\partial t}(0,t),\frac{\partial f}{\partial s} (0,t)\right) U_0(t), W_0(t)\metd dt\\
&=&\int_0^1 \meti \overline{R}\left( \frac{\partial f}{\partial s}(0,t),\frac{\partial f}{\partial t} (0,t)\right) U_0(t), W_0(t)\metd dt
\end{eqnarray*}
since $Z(1)=0$ by construction and so $g(1)=\meti Z(1),W_0(1)\metd=0$.
\end{proof}

\noindent {\bf Authors' Addresses:}

\vspace{.5cm}

\noindent A. J. Di Scala, \\ Dipartimento di Scienze Matematiche, Politecnico di Torino, \\
Corso Duca degli Abruzzi 24, 10129 Torino, Italy \\
\href{mailto:antonio.discala@polito.it}{antonio.discala@polito.it}\\
\url{http://calvino.polito.it/~adiscala/}

\vspace{.5cm}

\noindent F. Vittone, \\ Depto. de Matem\'atica, ECEN, FCEIA, \\
Universidad Nacional de Rosario, Rosario, Argentina\\
CONICET\\
\href{mailto:vittone@fceia.unr.edu.ar}{vittone@fceia.unr.edu.ar}\\

\end{document}